\documentclass[a4paper, 12pt]{article}
\usepackage[a4paper, total={6in,10in}]
{geometry}
\usepackage[utf8]{inputenc}
\usepackage{fullpage}
\usepackage{amsmath}
\usepackage{amssymb}
\usepackage{amsfonts}
\usepackage{amsthm}
\usepackage{graphicx}
\usepackage{cite}
\usepackage{booktabs}
\usepackage{float}
\usepackage{algorithm2e}
\usepackage{enumitem}
\newtheoremstyle{dotless}{}{}{\itshape}{}{\bfseries}{}{ }{}
\theoremstyle{dotless}
\newtheorem{theorem}{Theorem}[section]
\newtheorem{corollary}{Corollary}[section]
\newtheorem{lemma}{Lemma}[section]
\theoremstyle{remark}
\theoremstyle{dotless}
\newtheorem{remark}{\textit{\textbf{Remark}}}[section]
\theoremstyle{definition}

\title{Exploring new upper and lower bounds for the $A_{\alpha}$-energy of graphs}

\author{Mainak Basunia\thanks{Department of Mathematics, Indian Institute of Technology Kharagpur, Kharagpur 721302, India. Email: mainakmaths@iitkgp.ac.in, leo28mynnix@gmail.com}\and Pratima Panigrahi\thanks{Department of Mathematics, Indian Institute of Technology Kharagpur, Kharagpur 721302, India. Email: pratima@maths.iitkgp.ac.in}}

\date{}

\begin{document}
\maketitle
\baselineskip=0.25in

\begin{abstract}
Let $G$ be a graph on $n$ vertices and $m$ edges. For $\alpha \in [0,1]$, the $A_{\alpha}$-matrix of $G$ is defined as $A_{\alpha}(G) = \alpha D(G) + (1- \alpha) A(G)$, where $A(G)$ is the adjacency matrix and $D(G)$ is the degree diagonal matrix of $G$. If $\rho_1 \geq \rho_2 \ldots \geq \rho_n$ are the eigenvalues of $A_{\alpha}(G)$, the $A_{\alpha}$-energy of $G$ is defined as $E_{A_{\alpha}}(G) = \sum_{i=1}^{n} |\rho_i -\frac{2\alpha m}{n}|$. In this paper, we present novel upper and lower bounds for $E_{A_\alpha}(G)$ in terms of standard graph invariants, showing that each bound is sharp and identifying the specific graphs attaining them. For selected bounds, we provide brief comparative analysis with existing results, observing improved estimates. Furthermore, we establish new relations between $E_{A_\alpha}(G)$ and other well known graph energies, including adjacency, Laplacian, as well as the adjacency energy of the line graph.

\bigskip \noindent \textbf{Keywords:} Graphs, $A_{\alpha}$-matrix, $A_{\alpha}$-eigenvalues, $A_{\alpha}$-energy, bounds

\noindent {\bf AMS Subject Classification (2010):} 05C50, 05C05, 15A18
\end{abstract}


\section{Introduction}\label{sec1}
All graphs considered in this paper are simple, undirected and finite. Let $G=$ $(V(G)$, $E(G))$ be a graph with vertex set $V(G)$ and edge set $E(G)$. We denote the number of vertices (order) and edges (size) of $G$ by $n$ and $m$, respectively. The degree of the vertex $v$, denoted by $d_G(v)$, is the number of its neighbors in $G$. The first Zagreb index of $G$, denoted by $Z_1(G)$, is the sum of the squares of all vertex degrees.

The adjacency matrix $A(G)$, of $G$, is the $n \times n$ symmetric matrix whose $(i,j)^{th}$ entry is 1 if the $i^{th}$ and $j^{th}$ vertices are adjacent, and 0 otherwise. The degree matrix $D(G)$, of $G$, is the diagonal matrix of order $n$ whose diagonal entries are the vertex degrees. The Laplacian matrix and signless Laplacian matrix of $G$ are defined as $L(G) = D(G) - A(G)$ and $L_S(G) = D(G) + A(G)$, respectively. For any real $\alpha \in [0,1]$, Nikiforov \cite{A_Q-merging_by_Nikiforov} introduced the $A_{\alpha}$-matrix of $G$ as,
\begin{equation}\label{eq1}
A_{\alpha}(G) = \alpha D(G) + (1- \alpha) A(G),\qquad \alpha \in [0,1].
\end{equation}
Notable special cases include $A_{\alpha}(G) = A(G)$ when $\alpha = 0$, $A_{\alpha}(G) = \frac{1}{2}L_S(G)$ when $\alpha = \frac{1}{2}$, and $A_{\alpha}(G) = D(G)$ when $\alpha =1$. Related studies on the properties and spectrum of the $A_\alpha$-matrix can be found in \cite{note_on_positive_semidefiniteness_by_nikiforov, A_alpha_spec_of_graphs_by_lin_xue_shu, spectra_of_join_by_basunia}.

Let $M$ be a real symmetric matrix of order $p$. Its eigenvalues are real and can be arranged as: $\lambda_1(M) \geq \lambda_2(M) \geq \ldots \geq \lambda_p(M)$. The multiset of all the eigenvalues of $M$ is called the $M$-spectrum and is denoted by $\text{spec}(M)$. Furthermore, if $\lambda_1, \lambda_2, \ldots, \lambda_r$ are all the distinct eigenvalues of $M$ with corresponding multiplicities $m_1, m_2, \ldots, m_r$, then the spectrum of $M$ is denoted by $\text{spec}(M) = \{\lambda_1^{[m_1]}, \lambda_2^{[m_2]}, \ldots, \lambda_r^{[m_r]} \}$. For a graph $G$ on $n$ vertices, we denote the spectra of $A(G)$, $L(G)$, $L_S(G)$ and $A_\alpha(G)$ by $\gamma_1(G) \geq \gamma_2(G) \geq \ldots \geq \gamma_n(G)$, $\mu_1(G) \geq \mu_2(G) \geq \ldots \geq \mu_n(G)$, $q_1(G) \geq q_2(G) \geq \ldots \geq q_n(G)$ and $\rho_1(G) \geq \rho_2(G) \geq \ldots \geq \rho_n(G)$, respectively. When there is no confusion regarding the underlying graph, we just write the spectrum as  $\gamma_1 \geq \gamma_2 \geq \ldots \geq \gamma_n$ instead of $\gamma_1(G) \geq \gamma_2(G) \geq \ldots \geq \gamma_n(G)$ and so on. We also define $S_{A_\alpha}^{(k)}(G) =\sum_{i=1}^k\rho_i$, and denote its Laplacian and signless Laplacian analogues by $S_L^{(k)}(G)$ and $S_{L_S}^{(k)}(G)$, respectively.

The notion of the energy $\mathfrak{E}(G)$ of a graph $G$ with $n$ vertices and $m$ edges was introduced by Gutman \cite{energy_by_gutman} in connection with the $\pi$-molecular energy. It is defined as $\mathfrak{E}(G) = \sum_{i=i}^n |\gamma_i|$, whereas the Laplacian energy $E_L(G)$ \cite{lap_energy_by_gutman} and signless Laplacian energy $E_{L_S}(G)$ \cite{signless_lap_energy_by_abreu} are defined as $E_L(G) = \sum_{i=1}^{n}|\mu_i - \frac{2m}{n}|$ and $E_{L_S}(G) = \sum_{i=1}^{n}|q_i - \frac{2m}{n}|$, respectively. For $0\leq \alpha \leq 1$, the $A_\alpha$-energy \cite{A_alpha_energy_by_guo} of $G$ , denoted by $E_{A_\alpha}(G)$, is defined as $E_{A_\alpha}(G) = \sum_{i=1}^n |\rho_i - \frac{2\alpha m}{n}| = \sum_{i=1}^n \vartheta_i$, where $\vartheta_i =|\rho_i - \frac{2\alpha m}{n}|$ for $i=1,2, \ldots, n$. From this definition, it is clear that $E_{A_0}(G) = \mathfrak{E}(G)$ and $2E_{A_{\frac{1}{2}}}(G) =E_{L_S}(G)$ which shows that the $A_\alpha$-energy unifies the theories of (adjacency) energy and signless Laplacian energy of a graph. The study of $E_{A_\alpha}(G)$ is relatively recent, with contributions in \cite{sum_of_A_alpha_eigenvalues_by_lin,bounds_on_A_alpha_energy_by_zhou, On_max_a_alpha_spec_rad_of_unicyclic_bicyclic_fixed_girth_pendant_by_das_mahato, bounds_for_A_alpha_energy_by_shaban}. 

This paper advances the study of $A_{\alpha}$-energy by presenting new upper and lower bounds for it, expressed predominantly in terms of fundamental graph invariants, such as order, size, maximum/minimum degree and the first Zagreb index. We discuss the novelty of these bounds, show that each one is sharp, and identify the graphs achieving equality. For selected results, we conduct targeted numerical comparison with existing bounds in the literature, highlighting notable improvements in estimation quality. Beyond these bounds, we establish new relations between $E_{A_\alpha}(G)$ and other cornerstone graph energies, including adjacency energy, Laplacian energy, and the adjacency energy of the corresponding line graph, thereby creating a broader and more unified framework for energy based graph analysis. This contributions not only enrich the theoretical landscape but also serve, to some extent, as practical tools in situations where direct computation is difficult.

The remainder of the paper is organized as follows: Section $2$ covers preliminaries, basic notations and some relevant known results. Section $3$ and $4$ present our new upper and lower bounds for $E_{A_\alpha}(G)$, respectively. Section $5$ develops relationships between $E_{A_\alpha}(G)$ and other graph energies. Equality characterizations are provided throughout Sections $3-5$ whenever possible.


\section{Preliminaries}\label{sec2}
Throughout the remainder of this paper, while referring to a graph $G$, we will assume that $n$, $m$, $\Delta$, $\delta$ and $Z_1$ denote the order, size, maximum degree, minimum degree and the first Zagreb index of $G$, respectively, unless stated otherwise. For $0 \leq \alpha \leq 1$, the $A_{\alpha}$-spread of $G$, denoted by $\Theta_{A_\alpha}(G)$, is defined as the difference between the largest and smallest $A_\alpha$-eigenvalues of $G$, i.e. $\Theta_{A_\alpha}(G) = \rho_1 -\rho_n$. The line graph of $G$, represented as $\mathcal{L}(G)$, is the graph with $V(\mathcal{L}(G)) = E(G)$, and two vertices in $\mathcal{L}(G)$ are adjacent precisely when their corresponding edges in $G$ share a common endpoint. 

We adopt the following notations for standard graph classes : $P_k$ and $C_k$ denote the path and the cycle on $k$ vertices each, respectively; $K_a$ and $K_{a,b}$ to denote the complete graph on $a$ vertices and complete bipartite graph having two partite sets of $a$ and $b$ vertices, respectively; $S_a$ is the star graph $K_{1,a-1}$; $S_{a,b}$ is the double star obtained by joining the centers of $S_{a+1}$ and $S_{b+1}$; $W_k$ is the wheel graph formed by joining an isolated vertex to $C_{k-1}$; $L_k$ is the ladder graph with $k$ rungs, isomorphic to cartesian product of $P_k$ and $K_2$; $B_k$ is the book graph consisting of $k$ copies of $C_4$ sharing a common edge; $F_k$ to denote the friendship graph with $k$ triangles sharing one common vertex; and $C_{a,b}$ $(b\leq a)$ is the comb graph obtained from $P_a$ by attaching $b$ pendant vertices to $b$ consecutive vertices of $P_a$ starting from one end.   

In order to develop our main results, we rely on several foundational tools from the literature, which we summarize below. For any matrix $M$, $\mathcal{E}(M)$ denotes the matrix energy of $M$, which is defined as the sum of its singular values.

\begin{lemma}\textup{\cite{max_props_by_fan}}\label{lem9}
    Given two real square matrices $M$ and $N$ of same order, $\mathcal{E}(M+N) \leq \mathcal{E}(M) + \mathcal{E}(N)$.
    Equality is satisfied only when there is an orthogonal matrix $P$ that guarantees that both the matrices $PM$ and $PN$ are positive semidefinite.
\end{lemma}

\begin{lemma}\textup{\cite{max_props_by_fan}}\label{lem21} 
    If $M$ is a symmetric matrix of order $p$, then $\mathcal{E}(M) = \sum_{i=1}^{p}|\lambda_i(M)|$.
\end{lemma}

\begin{lemma} \textup{\cite{th_of_weyl_by_fan}} \label{lem1}
    For two symmetric matrices $M$ and $N$ of order $p$, $\sum_{i=1}^{k}\lambda_i(M+N) \leq \sum_{i=1}^{k}\lambda_i(M) + \sum_{i=1}^{k}\lambda_i(N)$, where $1 \leq k \leq p$.
\end{lemma}

\begin{lemma}\textup{\cite{A_alpha_energy_and_zagreb_by_pirzada}}\label{lem12}
    For a graph $G$, let $\sigma$ denotes the greatest integer in $[1,n]$ satisfying $\rho_{\sigma} \geq \frac{2\alpha m}{n}$. If $0\leq \alpha <1$, then $E_{A_\alpha}(G)=2S_{A_\alpha}^{(\sigma)}(G)-\frac{4\alpha m\sigma}{n}=$ \resizebox{.22\textwidth}{!}{$\displaystyle\max_{1\leq k \leq n}\bigg\{2S_{A_\alpha}^{(k)}(G)-\frac{4\alpha mk}{n}\bigg\}$}.
\end{lemma}

\begin{lemma}\textup{\cite{analytic_inequalities_by_mitrinovic}}\label{lem10}
    If $x_1, x_2, \ldots, x_p$ are real numbers such that $\sum_{i=1}^p|x_i| = 1$ and $\sum_{i=1}^p x_i = 0$, then \resizebox{.33\textwidth}{!}{$\displaystyle\big|\sum_{i=1}^p a_i x_i\big| \leq \frac{1}{2}\Big(\max_{1\leq i \leq p}a_i -\min_{1\leq i \leq p}a_i\Big)$},
    where $a_1, a_2, \ldots, a_p$ are real numbers.
\end{lemma}

\begin{lemma}\textup{\cite{commutativity_by_so}}\label{lem14}
    If $P$ and $Q$ both are $p\times p$ Hermitian matrices with $R=P+Q$, then 
    \begin{align*}
        &\lambda_i(R) \leq \lambda_j(P) + \lambda_{i-j+1}(Q), \quad 1 \leq j \leq i \leq p\\
        \text{and } \quad &\lambda_i(R) \geq \lambda_j(P) + \lambda_{i-j+p}(Q), \quad 1 \leq i \leq j \leq p.
    \end{align*}
    Equality in each inequality is satisfied if and only if for each of the three eigenvalues involved, there is a common eigenvector.
\end{lemma}

\begin{lemma}\textup{\cite{A_Q-merging_by_Nikiforov}}\label{lem17}
    Let $G$ be a connected graph with diameter d. If $0 \leq \alpha < 1$, then the number of distinct eigenvalues of $A_\alpha (G)$ is at least $d+1$.
\end{lemma}

\begin{lemma}\textup{\cite{bounds_on_A_alpha_spread_by_lin}}\label{lem22}
    Let $G$ be a graph with $n$ vertices and $m$ edges. If $0 \leq \alpha \leq 1$, then
    \begin{align*}
        \Theta_{A_\alpha}(G) \leq \sqrt{2\alpha^2 Z_1 + 4(1-\alpha)^2m - \frac{8\alpha^2 m^2}{n}}.
    \end{align*}
    Equality holds for $G \cong K_{\frac{n}{2}, \frac{n}{2}}$.
\end{lemma}


\section{Upper bounds for $A_{\alpha}$-energy of a graph}

We begin this section by obtaining an upper bound for $E_{A_{\alpha}}(G)$ that depends solely on the number of vertices and edges, making it both elementary and widely applicable.

\begin{theorem}\label{r11}
    Let $G$ be a graph on $n \geq 2$ vertices. If $0 \leq \alpha < 1$, then
    \begin{align}\label{eq105}
        E_{A_{\alpha}}(G) \leq 
        \begin{cases}
            4\alpha m\big(1-\frac{1}{n}\big), & \text{if $\alpha \geq \frac{n}{2(n-1)}$}.\\
            2m\big(1-\frac{2\alpha}{n}\big), & \text{if $\alpha < \frac{n}{2(n-1)}$}.
        \end{cases}
    \end{align}
    The graph $K_2$ satisfies the equality for all $\alpha \in [0,1)$.
\end{theorem}

\begin{proof}
    We take $V(G)=\{v_1, v_2, \ldots, v_n\}$. Let $G(e)$ be the spanning subgraph of $G$ containing only one edge $e=\{u,v\}$. If $A_\alpha(G(e))=(a_{ij})$ is the $A_\alpha$-matrix of $G(e)$, then
    \begin{align*}
        a_{ij}=
        \begin{cases}
            \alpha \quad &\text{ if } v_i=v_j=u \text{ or } v_i=v_j=v\\
            1-\alpha \quad &\text{ if } v_i=u, v_j=v \text{ or } v_i=v, v_j=u\\
            0 \quad &\text{ otherwise}.
        \end{cases}
    \end{align*}
    It is easy to see that $A_\alpha(G)=\sum_{e \in E(G)}A_\alpha(G(e))$. From Lemma \ref{lem21}, $E_{A_{\alpha}}(G) = \sum_{i=1}^n|\rho_i -\frac{2\alpha m}{n}|= \mathcal{E}\big(A_\alpha(G) - \frac{2\alpha m}{n}I_n\big)$. Therefore $E_{A_{\alpha}}(G) =\mathcal{E}\big(\sum_{e\in E(G)}A_\alpha(G(e)) - \frac{2\alpha m}{n}I_n\big)=\mathcal{E}\big(\sum_{e\in E(G)}\big[A_\alpha(G(e))-\frac{2\alpha}{n}I_n\big]\big).$
    Applying Lemma \ref{lem9}, we get $E_{A_{\alpha}}(G) \leq \sum_{e\in E(G)} \linebreak \mathcal{E}\big(A_\alpha(G(e))-\frac{2\alpha}{n}I_n\big)$.
    It is easy to see that for every $e\in E(G)$, the spectrum of the matrix $A_\alpha(G(e)) -\frac{2\alpha}{n}I_n$ is ${\big\{1-\frac{2\alpha}{n}}$, ${2\alpha-1-\frac{2\alpha}{n},-\frac{2\alpha}{n}^{[n-2]}\big\}}$. Therefore $E_{A_{\alpha}}(G) \leq m\big[|1-\frac{2\alpha}{n}| +|2\alpha-1-\frac{2\alpha}{n}| +|\frac{2\alpha}{n}|(n-2)\big]$, by using Lemma \ref{lem21}. For $0\leq \alpha <1$ and $n \geq 2$, $|1-\frac{2\alpha}{n}|=(1-\frac{2\alpha}{n})$, $|\frac{2\alpha}{n}|=\frac{2\alpha}{n}$ and \resizebox{.5\textwidth}{!}{$|2\alpha-1-\frac{2\alpha}{n}| =
        \begin{cases}
            \quad 2\alpha-1-\frac{2\alpha}{n}, \quad &\text{if } \alpha \geq \frac{n}{2(n-1)}\\
            -(2\alpha-1-\frac{2\alpha}{n}), \quad &\text{if } \alpha < \frac{n}{2(n-1)}
        \end{cases}$}. Using these in the last inequality and then simplifying it, we achieve the required bound.

        To verify the equality case, we observe that $\frac{n}{2(n-1)} =\frac{2}{2(2-1)}=1$ for $K_2$, and since $0 \leq \alpha < 1$, the right hand expression of \eqref{eq105} becomes $2m\big(1-\frac{2\alpha}{n}\big) = 2 \times 1 \times (1-\frac{2\alpha}{2})=2(1-\alpha)$. On the other hand, $\text{spec}(A_\alpha(K_2)) = \{1, {2\alpha -1}\}$, which yields $E_{A_\alpha}(K_2) = 2(1-\alpha)$. Thus $K_2$ satisfies the equality for all $\alpha \in [0,1)$.
\end{proof}

\begin{remark}\label{r103}
\rm{
    There is a well-established upper bound for the $A_\alpha$-energy of a graph \cite[Theorem $2.6$]{A_alpha_energy_and_zagreb_by_pirzada}. It shows that for $0\leq \alpha < 1$,
    \begin{align}\label{eq106}
        E_{A_\alpha}(G) \leq \sqrt{2(1-\alpha)^2mn + \alpha^2n\sum_{i=1}^{n}\Big(d_i - \frac{2m}{n}\Big)^2},
    \end{align}
    where $d_1 \geq d_2 \geq \ldots \geq d_n$ is the degree sequence of the graph $G$. Compared to \eqref{eq106}, Theorem \ref{r11} provides an upper bound that relies on fewer parameters, as it does not involve the degree sequence of the graph. Moreover, for various graphs like stars and double stars, numerical investigation indicates that Theorem \ref{r11} yields better bound than \eqref{eq106}. For example, for the star graph $S_{20}$ with $\alpha = 0.60$, Theorem \ref{r11} gives an upper bound of $43.32$, whereas \eqref{eq106} yields $48.35$. Likewise, for the double star $S_{12,21}$ with $\alpha = 0.70$, Theorem \ref{r11} provides an upper bounds of $92.48$ compared with $98.56$ from \eqref{eq106}.
}
\end{remark}

Next we present two lemmas on $S_{A_{\alpha}}^{(k)}(G)$. Together they yield a sharp upper bound for $E_{A_\alpha}(G)$, proved in the subsequent theorem. Throughout, if $k > n$, we interpret $S_{A_{\alpha}}^{(k)}(G)$ as $S_{A_{\alpha}}^{(n)}(G)$. 

\begin{lemma}\label{lem2}
    Let $G_1, G_2, \ldots, G_l$, $l \geq 1$, be edge disjoint subgraphs of $G$ with $E(G) = \bigcup_{i=1}^{l}E(G_i)$. Then for any $\frac{1}{2} \leq \alpha < 1$ and $1\leq k \leq n$,
    \begin{align*}
        S_{A_{\alpha}}^{(k)}(G) \leq \sum_{i=1}^{l} S_{A_{\alpha}}^{(k)}(G_i).
    \end{align*}
\end{lemma}

\begin{proof}
Let $|V(G_i)| = n_i$, $i=1,2, \ldots, l$. We construct the graph $G_i^\prime$ by adding $(n-n_i)$ isolated vertices to the graph $G_i$. Then $S_{A_{\alpha}}^{(k)}(G) = \sum_{j=1}^{k}\rho_j(G) = \sum_{j=1}^{k}\lambda_j(A_\alpha(G)) = \sum_{j=1}^{k}\lambda_j\Big(\sum_{i=1}^{l}A_\alpha(G_i^\prime)\Big)$. Using Lemma \ref{lem1}, $S_{A_{\alpha}}^{(k)}(G) \leq \sum_{i=1}^{l}\sum_{j=1}^{k}\lambda_j(A_\alpha(G_i^\prime)) =\linebreak \sum_{i=1}^{l}S_{A_{\alpha}}^{(k)}(G_i^\prime)$. When $\frac{1}{2} \leq \alpha < 1$, the $A_\alpha$-matrix for a graph is a positive semidefinite matrix i.e., all the $A_\alpha$-eigenvalues are non-negative. Then $S_{A_{\alpha}}^{(k)}(G_i^\prime) =S_{A_{\alpha}}^{(k)}(G_i)$, $i=1, 2, \ldots, l$ and hence the result follows. \qedhere
\end{proof}

\begin{lemma}\label{r4}
    Let $G$ be a graph. If $\frac{1}{2} \leq \alpha < 1$, then for $k=1,2, \ldots, n,$
    \begin{align*}
        \resizebox{.95\textwidth}{!}{$S_{A_{\alpha}}^{(k)}(G) \leq \frac{1}{2}\bigg( \alpha(4m -3\Delta +2k -1) + \sqrt{\alpha^2 (\Delta +1)^2 + 4\Delta (1-2\alpha)} \bigg) -\left\lfloor \frac{1}{k}\right\rfloor (2\alpha -1)(m-\Delta)$}.
    \end{align*}
\end{lemma}

\begin{proof}
    Maximum degree of $G$ is $\Delta$, therefore $K_{1,\Delta}$ is a subgraph of $G$. Applying Lemma \ref{lem2} on $G$, we have for $\frac{1}{2} \leq \alpha < 1$, $S_{A_{\alpha}}^{(k)}(G) \leq S_{A_{\alpha}}^{(k)}(K_{1,\Delta}) + (m-\Delta) S_{A_{\alpha}}^{(k)}(K_2)$, $k=1,2,\ldots,n$. From \cite{A_Q-merging_by_Nikiforov}, $\text{spec}(A_\alpha(K_{1,\Delta}))=\Big\{ \frac{1}{2}\Big(\alpha(\Delta +1) + \sqrt{\alpha^2(\Delta +1)^2 +4\Delta(1-2\alpha)}\Big), \alpha^{[\Delta -1]}, \frac{1}{2}\Big(\alpha(\Delta +1)-\sqrt{\alpha^2(\Delta +1)^2 +4\Delta(1-2\alpha)}\Big)\Big\}$. When $k=1,2, \ldots, \Delta$, $S_{A_{\alpha}}^{(k)}(K_{1,\Delta}) = \frac{1}{2}\Big(\alpha(\Delta +1) + \linebreak \sqrt{\alpha^2(\Delta +1)^2 +4\Delta(1-2\alpha)}\Big) + (k-1)\alpha$ and when $k=\Delta +1,\Delta +2, \ldots, n$, $S_{A_{\alpha}}^{(k)}(K_{1,\Delta}) \leq \frac{1}{2}\Big(\alpha(\Delta +1) + \sqrt{\alpha^2(\Delta +1)^2 +4\Delta(1-2\alpha)}\Big) + (k-1)\alpha$, because we know $\alpha \geq \frac{1}{2}\Big(\alpha(\Delta +1) -\sqrt{\alpha^2(\Delta +1)^2 +4\Delta(1-2\alpha)}\Big)$. Therefore for any $1\leq k \leq n$, $S_{A_{\alpha}}^{(k)}(K_{1,\Delta}) \leq \frac{1}{2}\Big(\alpha(\Delta +1) + \sqrt{\alpha^2(\Delta +1)^2 +4\Delta(1-2\alpha)}\Big) + (k-1)\alpha$, $\frac{1}{2} \leq \alpha < 1$.

    Again from \cite{A_Q-merging_by_Nikiforov}, we have $\text{spec}(A_\alpha(K_2)) = \{1, {2\alpha -1}\}$. Thus $S_{A_{\alpha}}^{(k)}(K_2) = 1+\big(1-\left\lfloor \frac{1}{k} \right\rfloor\big)(2\alpha -1)$, $k=1,2,\ldots,n$. Using this and the last relation involving $S_{A_{\alpha}}^{(k)}(K_{1,\Delta})$ into $S_{A_{\alpha}}^{(k)}(G) \leq S_{A_{\alpha}}^{(k)}(K_{1,\Delta}) + (m-\Delta) S_{A_{\alpha}}^{(k)}(K_2)$, we get $S_{A_{\alpha}}^{(k)}(G) \leq \frac{1}{2}\Big(\alpha(\Delta +1) + \linebreak \sqrt{\alpha^2(\Delta +1)^2 +4\Delta(1-2\alpha)}\Big) + (k-1)\alpha + (m-\Delta)\Big(1+\big(1-\left\lfloor \frac{1}{k} \right\rfloor\big)(2\alpha -1)\Big)$, $\frac{1}{2} \leq \alpha < 1$ and $k=1,2, \ldots, n$. The result follows from here just by rearranging the terms.
\end{proof}

\begin{theorem}\label{r12}
    Let $G$ be a connected graph with $n \geq 2$. If $\frac{1}{2} < \alpha < 1$, then
    \begin{align}\label{eq101}
         E_{A_\alpha}(G) \leq \alpha (4m -3\Delta - \frac{4m}{n} +1) + \sqrt{\alpha^2(\Delta +1)^2 +4(1-2\alpha)\Delta}.
    \end{align}
    Equality is satisfied if and only if $G \cong K_{1,n-1}.$
\end{theorem}

\begin{proof}
    Using the upper bound for $S_{A_\alpha}^{(\sigma)}(G)$ from Lemma \ref{r4} in the relation $E_{A_\alpha}(G)=2S_{A_\alpha}^{(\sigma)}(G) -\frac{4\alpha m\sigma}{n}$ (Lemma \ref{lem12}), for $\frac{1}{2} < \alpha <1$ we get
    \begin{align}\label{eq100}
        E_{A_{\alpha}}(G) &\leq \resizebox{.87\textwidth}{!}{$\alpha(4m -3\Delta +2\sigma -1) + \sqrt{\alpha^2 (\Delta +1)^2 + 4(1-2\alpha)\Delta} -2\left\lfloor \frac{1}{\sigma}\right\rfloor (2\alpha -1)(m-\Delta) -\frac{4\alpha m\sigma}{n}$}\nonumber\\
        &= \resizebox{.87\textwidth}{!}{$\alpha(4m-3\Delta-1) + \sqrt{\alpha^2 (\Delta +1)^2 + 4(1-2\alpha)\Delta} -2\left\lfloor \frac{1}{\sigma}\right\rfloor (2\alpha -1)(m-\Delta) -2\alpha \sigma (\frac{2m}{n}-1)$}.
    \end{align}
    Note that $\sigma \geq 1$. For $\frac{1}{2} < \alpha < 1$, we have $2\left\lfloor \frac{1}{\sigma}\right\rfloor (2\alpha -1)(m-\Delta) \geq 0$ as $m \geq \Delta$. Also, for a connected graph, $m \geq n-1$, and hence $2m \geq n$, which implies $2\alpha \sigma (\frac{2m}{n}-1) \geq 0$. Moreover $2\alpha \sigma (\frac{2m}{n}-1) \geq 2\alpha(\frac{2m}{n} -1)$. Substituting this along with $2\left\lfloor \frac{1}{\sigma}\right\rfloor (2\alpha -1)(m-\Delta) \geq 0$ in \eqref{eq100}, we obtain $E_{A_{\alpha}}(G) \leq \alpha(4m-3\Delta-1) + \sqrt{\alpha^2 (\Delta +1)^2 + 4(1-2\alpha)\Delta} -2\alpha\big(\frac{2m}{n} -1\big) =\alpha (4m -3\Delta - \frac{4m}{n} +1) + \sqrt{\alpha^2(\Delta +1)^2 +4(1-2\alpha)\Delta}.$
    
    Next we verify the equality condition. Consider the star graph $K_{1,n-1}$, where $m=\Delta=n-1$. Substituting these into $\eqref{eq101}$, the right hand side becomes $\alpha\big(n-4+\frac{4}{n}\big) + \sqrt{\alpha^2n^2 + 4(1-2\alpha)(n-1)}$. From \cite{A_Q-merging_by_Nikiforov}, the spectrum of $A_{\alpha}(K_{1,n-1})$ is : $\Big\{\frac{1}{2}\big(\alpha n + \linebreak \sqrt{\alpha^2n^2 + 4(1-2\alpha)(n-1)}\big),\alpha^{[n-2]},$ $\frac{1}{2}\big(\alpha n -$ $ \sqrt{\alpha^2n^2 + 4(1-2\alpha)(n-1)}\big)\Big\}$. Applying this and $\frac{2\alpha m}{n}=2\alpha -\frac{2\alpha}{n}$ into $E_{A_\alpha}(K_{1,n-1}) = \sum_{i=1}^{n} \big| \rho_i(K_{1,n-1}) - \frac{2\alpha m}{n} \big|$ lead us to obtain $E_{A_\alpha}(K_{1,n-1})$ = $\alpha\big(n-4+\frac{4}{n}\big) + \sqrt{\alpha^2n^2 + 4(1-2\alpha)(n-1)}$. Thus $K_{1,n-1}$ satisfies the equality for any $\frac{1}{2} < \alpha < 1$.

    Now consider the reverse direction. We assume that a connected graph $G^*$ satisfies the equality in \eqref{eq101}. Therefore 
    \begin{align}\label{eq102}
         E_{A_\alpha}(G^*) & =  \alpha (4m -3\Delta - \frac{4m}{n} +1) + \sqrt{\alpha^2(\Delta +1)^2 +4(1-2\alpha)\Delta}\nonumber\\
         & =\alpha(4m-3\Delta-1) + \sqrt{\alpha^2 (\Delta +1)^2 + 4(1-2\alpha)\Delta} -2\alpha \Big(\frac{2m}{n}-1\Big).
    \end{align}
    $G^*$ satisfies \eqref{eq100} also. If \eqref{eq100} and \eqref{eq102} both hold simultaneously for $G^*$, then $G^*$ must satisfy $\sigma =1$ and $2\left\lfloor \frac{1}{\sigma}\right\rfloor (2\alpha -1)(m-\Delta) = 0$ both. Since $\frac{1}{2} < \alpha < 1$, this implies $m =\Delta$ for $G^*$. If $d_1 \geq d_2 \geq \ldots \geq d_n$ is the degree sequence of $G^*$, then $2m=\sum_{i=1}^{n}d_i = \Delta + \sum_{i=2}^{n}d_i$. As $m=\Delta$, it becomes $\Delta =\sum_{i=2}^{n}d_i$. Since $G^*$ is connected, each $d_i \geq 1$ for $i=2, 3, \ldots, n$. So $\Delta \geq \sum_{i=2}^{n}1$, giving $\Delta \geq n-1$. On the other hand, in any graph, the maximum degree $\Delta \leq n-1$. Hence $\Delta = n-1$, and so $n-1 =\sum_{i=2}^{n}d_i$, which is the sum of $n-1$ positive integers, each at least $1$. This forces the degree sequence of connected graph $G^*$ to be $\{n-1, 1, 1, \ldots, 1 \}$, implying that $K_{1,n-1}$ becomes the only candidate for $G^*$.
    
    Thus, equality holds in \eqref{eq101} if and only if $G \cong K_{1, n-1}$, for any $\alpha \in (\frac{1}{2}, 1)$, which completes the proof.
\end{proof}


\section{Lower bounds for $A_{\alpha}$-energy of a graph}

This section presents two new analytic lower bounds for the $A_{\alpha}$-energy of a graph. Both bounds are obtained via spectral relations combined with classical inequality techniques. We begin with Theorem \ref{r13}, which is exact for certain extremal families and forms the basis for a detailed comparative analysis with known bounds from the literature.

\begin{theorem}\label{r13}
    Let $0 \leq \alpha < 1$. Then for a graph $G$,
    \begin{align}\label{eq103}
         E_{A_\alpha}(G) \geq \frac{2}{\Theta_{A_\alpha}(G)}\bigg[\alpha^2Z_1 +2(1-\alpha)^2m -\frac{4\alpha^2 m^2}{n}\bigg].
    \end{align}
    The equality holds for the graphs $K_n$ and $K_{\frac{n}{2}, \frac{n}{2}}$.
\end{theorem}

\begin{proof}
    For $0 \leq \alpha < 1$, let us consider $a_i = \rho_i$ and $x_i = \frac{\rho_i - \frac{2\alpha m}{n}}{\sum_{i=1}^{n}|\rho_i - \frac{2\alpha m}{n}|}$, $i=1,2,\ldots,n.$ Then $\sum_{i=1}^n x_i = \frac{\sum_{i=1}^n (\rho_i - \frac{2\alpha m}{n})}{\sum_{i=1}^{n}|\rho_i - \frac{2\alpha m}{n}|}=\frac{2\alpha m -2\alpha m}{\sum_{i=1}^{n}|\rho_i - \frac{2\alpha m}{n}|}=0$ and $\sum_{i=1}^n |x_i| = \frac{\sum_{i=1}^n |\rho_i - \frac{2\alpha m}{n}|}{\sum_{i=1}^{n}|\rho_i - \frac{2\alpha m}{n}|}=1$. Now applying Lemma \ref{lem10}, we get $\Big|\sum_{i=1}^n \frac{\rho_i(\rho_i -\frac{2\alpha m}{n})}{\sum_{i=1}^{n}|\rho_i - \frac{2\alpha m}{n}|}\Big| \leq \displaystyle\frac{1}{2}\Big(\max_{1\leq i \leq n}\rho_i -\min_{1\leq i \leq n}\rho_i\Big)=\frac{1}{2}(\rho_1 -\rho_n)$, i.e. $\frac{1}{E_{A_\alpha}(G)} \big|\sum_{i=1}^n {\rho_i}^2 -\frac{2\alpha m}{n}\sum_{i=1}^n \rho_i\big| \leq \frac{1}{2}\Theta_{A_\alpha}(G)$. Applying $\sum_{i=1}^n {\rho_i}^2=\alpha^2 Z_1 + 2(1-\alpha)^2m$ and $\sum_{i=1}^n \rho_i=2\alpha m$ from \cite{A_Q-merging_by_Nikiforov}, we get $E_{A_\alpha}(G) \geq \frac{2}{\Theta_{A_\alpha}(G)}\big|\alpha^2\big(Z_1-\frac{4m^2}{n}\big) +2(1-\alpha)^2m\big|$. We know from \cite{upper_bound_on_sum_square_degree_by_caen} that $Z_1 \geq \frac{4m^2}{n}$ with equality if and only if $G$ is regular. Thus, the quantity inside the modulus sign being positive, we can remove the modulus and get the required inequality.
    
    To prove the equality, we know from \cite{A_Q-merging_by_Nikiforov} that $\text{spec}(A_\alpha(K_n))=\big\{n-1, {\alpha n-1}^{[n-1]}\big\}$ and $\text{spec}\big(A_\alpha\big(K_{\frac{n}{2}, \frac{n}{2}}\big)\big)=\big\{\frac{n}{2}, \frac{\alpha n}{2}^{[n-2]}, \alpha n-\frac{n}{2}\big\}$. Using these information, it can be easily seen that the graphs $K_n$ and $K_{\frac{n}{2}, \frac{n}{2}}$ satisfy the desired equality.
\end{proof}

In Remark \ref{r101} below, we discuss the strength and effectiveness of our result given in Theorem \ref{r13}, in comparison with the lower bounds due to Zhou et al.\cite[Theorem $1.1$ and Theorem $1.2$]{bounds_on_A_alpha_energy_by_zhou}, which are stated as follows :

\begin{lemma}\textup{\cite{bounds_on_A_alpha_energy_by_zhou}}\label{lem23}
    Let $G$ be graph and $\frac{1}{2} \leq \alpha < 1$. For $\rho_1 \geq \rho_2 \geq \ldots \geq \rho_n$ as eigenvalues of $A_{\alpha}(G)$, we write $\vartheta_i =|\rho_i - \frac{2\alpha m}{n}|$ for $i=1,2, \ldots, n$. Arranging the $\vartheta_i$'s in non decreasing order, we rename them as $\xi_1 \geq \xi_2 \geq \ldots \geq \xi_n \geq 0$. Then
    \begin{enumerate}[label={\upshape (\alph*)}]
        \item If $\xi_n > 0$, $E_{A_\alpha}(G) \geq 2\sqrt{\big[\alpha^2 Z_1 + 2(1-\alpha)^2m -\frac{4\alpha^2 m^2}{n}\big]n}\cdot \frac{\sqrt{\xi_1 \xi_n}}{\xi_1 +\xi_n}$. Equality holds if and only if $G \cong \frac{n}{2}K_2$ or $gK_{\frac{2m}{n}+1} \bigcup h(K_{\frac{2m}{n}+1, \frac{2m}{n}+1} \setminus F)$, where $g$ and $h$ are some non-negative integers, $\frac{2m}{n} \geq 2$ is an integer, and $F$ is a perfect matching of $K_{\frac{2m}{n}+1, \frac{2m}{n}+1}$.
        
        \item If $\xi_n = 0$, $E_{A_\alpha}(G) \geq \frac{\alpha^2 Z_1 + 2(1-\alpha)^2m -\frac{4\alpha^2 m^2}{n}}{\xi_1}$. Equality holds if and only if $G \cong K_{\frac{n}{2}, \frac{n}{2}}$.
    \end{enumerate}
\end{lemma}


\begin{remark}\label{r101}
\rm{
    Compared to Lemma \ref{lem23}, Theorem $\ref{r13}$ has several structural advantages: it is valid for the entire range $0 \leq \alpha <1$, has a single unified formula with no branch conditions, and in terms of spectral quantity, it depends only on the $A_{\alpha}$-spectral spread $\Theta_{A_\alpha}(G)$ (requiring just $\rho_1$ and $\rho_n$). In contrast, Lemma $\ref{lem23}$ applies only for $\frac{1}{2} \leq \alpha <1$, has separate cases for $\xi_n \geq 0$ and $\xi_n =0$, and requires the computation and ordering of all deviation terms to identify $\xi_1$ and $\xi_n$. These features make our result easier to apply and more broadly applicable. Moreover, we conducted a numerical investigation over diverse graph families, a glimpse of which is shown in Table \ref{tab:lower bound-comparison}. It has been observed that in most of the cases the bound in Theorem \ref{r13} outperforms that in Lemma \ref{lem23}, especially for large and irregular graphs.
    
    \begin{table}[!h]
    \centering
        \footnotesize
        \begin{tabular}{|c|c|c|c|c|}
            \hline \textbf{Graph} & $\boldsymbol{\alpha}$ & \textbf{Theorem $\ref{r13}$ Bound} & \textbf{Lemma $\ref{lem23}$ Bound} & \textbf{Better Bound} \\
            \hline 
            $S_4$  & $0.50$ & $2.25$ & $2.24$ & Theorem $\ref{r13}$ \\
            $F_3$  & $0.60$ & $4.89$ & $2.20$ & Theorem $\ref{r13}$ \\
            $C_9$  & $0.70$ & $2.78$ & $2.71$ & Theorem $\ref{r13}$ \\
            $B_4$  & $0.80$ & $7.00$ & $6.43$ & Theorem $\ref{r13}$ \\
            $P_{10}$ & $0.90$ & $2.68$ & $0.76$ & Theorem $\ref{r13}$ \\
            \hline
        \end{tabular}
        \caption{Numerical comparison of our result (Theorem $\ref{r13}$) and Lemma $\ref{lem23}$ for different graphs and $\alpha$ values.}\label{tab:lower bound-comparison}
    \end{table}

    In addition, several other lower bounds from the literature such as Theorem $3.3$ from \textup{\cite{A_alpha_energy_and_zagreb_by_pirzada}} were also tested numerically against Theorem $\ref{r13}$. Not only did they fail to outperform our bound in most of the cases, but in many instances they even underperformed relative to Lemma $\ref{lem23}$ itself. This reinforces the strength, versatility, and competitiveness of Theorem $\ref{r13}$ as a reliable lower bound for the $A_{\alpha}$-energy.
    }
\end{remark}

The next result directly follows from Theorem $\ref{r13}$ upon substituting the $A_{\alpha}$-spectral spread $\Theta_{A_\alpha}(G)$ by its upper bound in terms of basic graph invariants given in Lemma $\ref{lem22}$. While this substitution renders the bound more computationally convenient, it comes at the cost of some sharpness compared to the original form in Theorem $\ref{r13}$.

\begin{corollary}\label{r16}
    Let $0 \leq \alpha < 1$. Then for a graph $G$,
    \begin{align*}
         E_{A_\alpha}(G) \geq \sqrt{2\alpha^2 Z_1 + 4(1-\alpha)^2m -\frac{8\alpha^2m^2}{n}}.
    \end{align*}
    Equality holds for $G \cong K_{\frac{n}{2},\frac{n}{2}}$.
\end{corollary}

Next we establish an alternative bound which retains sharpness for several extremal graph classes.

\begin{theorem}\label{r17}
    Let $G$ be a connected graph. Then
    \begin{enumerate}[label={\upshape (\alph*)}]
        \item If $0 \leq \alpha \leq \frac{1}{2}$, then $E_{A_\alpha}(G) \geq 2\big[ \frac{(1-\alpha)Z_1}{m} - \frac{2\alpha m}{n} + (2\alpha -1)\Delta \big]$. Equality holds for the graphs $K_n$ and $K_{\frac{n}{2}, \frac{n}{2}}$ for all $0 \leq \alpha \leq \frac{1}{2}$, and $K_{1,n-1}$ for $\alpha =\frac{1}{2}$.

        \item If $\frac{1}{2} < \alpha < 1$, then $E_{A_\alpha}(G) \geq \frac{2\alpha Z_1}{m} + \frac{4(1-3\alpha) m}{n}$. Equality holds for the graphs $K_n$ and $K_{\frac{n}{2}, \frac{n}{2}}$ for all $\frac{1}{2} < \alpha < 1$.
    \end{enumerate}
\end{theorem}

\begin{proof}
    \begin{enumerate}[label=\upshape (\alph*), listparindent=\parindent, parsep=0pt]
        \item From \cite{bounds_for_A_alpha_spec_rad_by_alhevaz}, we have $\rho_1(G) \geq (1-\alpha)q_1(G) + (2\alpha -1)\Delta$ for $0\leq \alpha \leq \frac{1}{2}$. It is known \cite{energy_of_line_graphs_by_gutman_robbiano} that $q_1(G) = 2 + \gamma_1(\mathcal{L}(G))$, where $\mathcal{L}(G)$ is the line graph of $G$. From \cite{A_Q-merging_by_Nikiforov}, we know that for any graph $G$, $\gamma_1(G)\geq\frac{2m}{n}$. Therefore for the graph $\mathcal{L}(G)$, we get $\gamma_1(\mathcal{L}(G)) \geq \frac{ 2\times|E(\mathcal{L}(G))|}{ |V(\mathcal{L}(G))|}=\frac{Z_1-2m}{m}$. Therefore $q_1(G)\geq 2+\frac{Z_1-2m}{m}$ and using this, from above we get $\rho_1(G) \geq (1-\alpha)\big(2+\frac{Z_1-2m}{m}\big) + (2\alpha -1)\Delta = \frac{(1-\alpha)Z_1}{m} + (2\alpha -1)\Delta$. From Lemma \ref{lem12}, $E_{A_\alpha}(G)=$ \resizebox{.22\textwidth}{!}{$\displaystyle\max_{1\leq k \leq n}\bigg\{2S_{A_\alpha}^{(k)}(G)-\frac{4\alpha mk}{n}\bigg\}$} $\geq 2S_{A_\alpha}^{(1)}(G)-\frac{4\alpha m}{n}= 2\big[\rho_1(G) - \frac{2\alpha m}{n}\big] \geq 2\big[ \frac{(1-\alpha)Z_1}{m} + (2\alpha -1)\Delta - \frac{2\alpha m}{n} \big]$, the last inequality comes after using the lower bound for $\rho_1(G)$ we obtained.

        To talk about the equality cases, using $\text{spec}(A_\alpha(K_n))=\big\{n-1, {\alpha n-1}^{[n-1]}\big\}$ and $\frac{2\alpha m}{n} =\alpha (n-1)$, we observe that $E_{A_{\alpha}}(K_n) = |n-1-\alpha (n-1)| + (n-1) |\alpha n -1 -\alpha (n-1)| = 2(1-\alpha)(n-1)$, while using $Z_1 = n(n-1)^2$ and $\Delta = n-1$ for $K_n$, the right hand side becomes $2\Big[ \frac{(1-\alpha) n(n-1)^2}{\frac{n(n-1)}{2}} - \frac{2\alpha n(n-1)}{2n} + (2\alpha -1)(n-1) \Big]$, simplifying which gives $2(1-\alpha)(n-1)$, same as $E_{A_{\alpha}}(K_n)$. Thus $K_n$ satisfies the equality for $0 \leq \alpha \leq \frac{1}{2}$. Similarly for the graph $K_{\frac{n}{2}, \frac{n}{2}}$, using $\text{spec}(A_\alpha(K_{\frac{n}{2}, \frac{n}{2}}))= \big\{ \frac{n}{2}, \frac{\alpha n}{2}^{[n-2]}, \alpha n -\frac{n}{2} \big\}$ and $\frac{2\alpha m}{n} = \frac{\alpha n}{2}$, we get $E_{A_{\alpha}}(K_{\frac{n}{2}, \frac{n}{2}}) = |\frac{n}{2}-\frac{\alpha n}{2}| + (n-2) |\frac{\alpha n}{2} - \frac{\alpha n}{2}| + |\alpha n -\frac{n}{2} - \frac{\alpha n}{2}| = (1-\alpha)n$. Also, $Z_1 =\frac{n^3}{4}$ and $\Delta = \frac{n}{2}$ make the right hand side $2\big[ \frac{(1-\alpha)\frac{n^3}{4}}{\frac{n^2}{4}} - \frac{2\alpha \frac{n^2}{4}}{n} + (2\alpha -1)\frac{n}{2} \big]$, after simplification which turns out to become $(1-\alpha)n$. Hence $K_{\frac{n}{2}, \frac{n}{2}}$ too satisfies the equality for $0 \leq \alpha \leq \frac{1}{2}$. For the star $K_{1,n-1}$, we use $\text{spec}(A_\alpha(K_{1, n-1}))= \big\{ \frac{1}{2}\big( \alpha n + \sqrt{\alpha^2 n^2 + 4(1-2\alpha)(n-1)} \big), \alpha^{[n-2]}, \frac{1}{2}\big( \alpha n - \sqrt{\alpha^2 n^2 + 4(1-2\alpha)(n-1)} \big) \big\}$, $\frac{2\alpha m}{n} = \frac{2\alpha(n-1)}{n}$, $Z_1 = (n-1)^2 + (n-1)$ and $\Delta = n-1$ particularly for $\alpha=\frac{1}{2}$, and similar to as we did for $K_n$ and $K_{\frac{n}{2}, \frac{n}{2}}$, eventually we obtain that $E_{A_{\frac{1}{2}}}(K_{1,n-1})$ = right hand side quantity of the inequality = $n-2+\frac{n}{2}$. Thus the equality holds for $K_{1, n-1}$ too, when $\alpha = \frac{1}{2}$ .\label{i1}
        
        \item  Let $\frac{1}{2} < \alpha < 1$. Then from \cite{bounds_for_A_alpha_spec_rad_by_alhevaz}, $\rho_1(G) \geq \alpha q_1(G) + (1-2\alpha)\gamma_1(G)$. As seen in the proof of part \ref{i1}, $q_1(G) \geq 2+\frac{Z_1 - 2m}{m}$ and $\gamma_1(G) \geq \frac{2m}{n}$. Using these, the first inequality produces $\rho_1(G) \geq \alpha \big(2+\frac{Z_1 - 2m}{m}\big) + (1-2\alpha)\times \frac{2m}{n}=\frac{\alpha Z_1}{m} + \frac{2(1-2\alpha )m}{n}$. By plugging this into $E_{A_\alpha}(G) \geq 2\big[\rho_1(G) -\frac{2\alpha m}{n} \big]$, we are able to obtain the desired inequality.

        The equality cases follow in the manner similar to that of part \ref{i1}.
        \label{i2} \qedhere
    \end{enumerate}
\end{proof}


\section{Relations between $A_\alpha$-energy and other graph energies associated with a graph}

This section demonstrates how $E_{A_\alpha}(G)$ is related to other graph energies, like $\mathfrak{E}(G)$, $E_L(G)$, as well as $\mathfrak{E}(\mathcal{L}(G))$. We begin with the following relation.

\begin{theorem}\label{r20}
    If $G$ is a connected graph with $n \geq 2$ and $\zeta$ is the adjacency rank of $G$, then for $\frac{1}{2} \leq \alpha < 1$,
    \begin{align}\label{eq18}
        E_{A_\alpha}(G) + \alpha E_L(G) \geq 2\mathfrak{E}(G) -\frac{4\alpha m\zeta}{n}.
    \end{align}
    Equality holds if and only if $\alpha =\frac{1}{2}$ and $G \cong K_{\frac{n}{2}, \frac{n}{2}}$.
\end{theorem}

\begin{proof}
    Let $\zeta^+$ and $\zeta^-$ be the numbers of the positive and the negative eigenvalues (including multiplicities) of $A(G)$, respectively. Therefore $1\leq \zeta^+\leq n-1$, $1\leq \zeta^-\leq n-1$ and $\zeta=\zeta^+ + \zeta^-$. From Lemma \ref{lem12},
    \begin{align}\label{eq14}
        \resizebox{.55\textwidth}{!}{$E_{A_\alpha}(G) =\displaystyle\max_{1\leq k \leq n}\bigg\{2S_{A_\alpha}^{(k)}(G)-\frac{4\alpha mk}{n}\bigg\} \geq 2 \sum_{i=1}^{\zeta^+}\rho_i-\frac{4\alpha m\zeta^+}{n}$}
    \end{align}
    Analogous to Lemma \ref{lem12}, for $E_L(G)$, it is a well known result from \cite{laplacian_energy_of_graphs_by_kcdas} that $E_L(G) =$ \resizebox{.22\textwidth}{!}{$\displaystyle\max_{1\leq k \leq n}\bigg\{2\sum_{i=1}^{k}\mu_i-\frac{4mk}{n}\bigg\}$}.
    Putting $k=\zeta^-$,
    \begin{align}\label{eq15}
        \resizebox{.3\textwidth}{!}{$E_L(G) \geq 2\sum_{i=1}^{\zeta^-}\mu_i-\frac{4m\zeta^-}{n}$}.
    \end{align}
    Now $A_\alpha(G) =\alpha D(G) + (1-\alpha)A(G)=\alpha\big(D(G)-A(G)\big) + A(G)=\alpha L(G) + A(G)$.
    Applying Lemma \ref{lem14}, we get $\rho_n \leq \alpha\mu_i + \gamma_{n-i+1}$, $1 \leq i \leq n$ and $\rho_i \geq \alpha\mu_n + \gamma_{i}$, $1 \leq i \leq n$. We know that $\mu_n=0$ and since for $\frac{1}{2} \leq \alpha <1$, $A_\alpha(G)$ is a positive semidefinite matrix, $\rho_n \geq 0$. Therefore from the above inequalities, for $\frac{1}{2} \leq \alpha <1$, we get    
    \begin{align}
        \alpha\mu_i &\geq -\gamma_{n-i+1}, \quad 1 \leq i \leq n \label{eq16} \\
        \text{and } \quad\quad \rho_i &\geq \gamma_{i}, \quad 1 \leq i \leq n \label{eq17}.
    \end{align}
    Using \eqref{eq17} in \eqref{eq14}, we get
    $E_{A_\alpha}(G) \geq 2 \sum_{i=1}^{\zeta^+}\gamma_i-\frac{4\alpha m\zeta^+}{n}$, $\frac{1}{2} \leq \alpha <1$
    and using \eqref{eq16} in \eqref{eq15}, we get $\alpha E_L(G) \geq -2\sum_{i=1}^{\zeta^-}\gamma_{n-i+1}-\frac{4\alpha m\zeta^-}{n}$, $\frac{1}{2} \leq \alpha <1$.
    Adding these two inequalities, for $\frac{1}{2} \leq \alpha <1$. we get $E_{A_\alpha}(G) + \alpha E_L(G) \geq 2\Big[\sum_{i=1}^{\zeta^+}\gamma_i -\sum_{j=1}^{\zeta^-}\gamma_{n-j+1}\Big]-\frac{4\alpha m}{n}(\zeta^+ + \zeta^-)=2\Big[\sum_{i=1}^{\zeta^+}|\gamma_i| +\sum_{j=1}^{\zeta^-}|\gamma_{n-j+1}|\Big]-\frac{4\alpha m\zeta}{n}=2\mathfrak{E}(G)-\frac{4\alpha m\zeta}{n}.$
    Thus we get the required inequality.\bigskip
    
    Now we discuss the equality of \eqref{eq18}. Putting $\alpha = \frac{1}{2}$ in \eqref{eq18}, for a bipartite graph $G$, we get $E_{A_{\frac{1}{2}}}(G) + \frac{1}{2}E_{L_S}(G) \geq 2\mathfrak{E}(G)-\frac{2m\zeta}{n}$,
    as the signless Laplacian spectrum and the Laplacian spectrum of a bipartite graph are identical. Therefore $4E_{A_{\frac{1}{2}}}(G) \geq 4\mathfrak{E}(G)-\frac{4m\zeta}{n}$. Now $K_{\frac{n}{2}, \frac{n}{2}}$ is a bipartite graph. Using $\text{spec}\big(A_\alpha\big(K_{\frac{n}{2}, \frac{n}{2}}\big)\big)=\big\{\frac{n}{2}, \frac{\alpha n}{2}^{[n-2]}, \alpha n-\frac{n}{2}\big\}$ for $\alpha=0, \frac{1}{2}$, we can easily verify that the equality is true for $G \cong K_{\frac{n}{2}, \frac{n}{2}}$. Conversely let the equality be hold for a connected graph $G$ and for some $\alpha$ ($\frac{1}{2} \leq \alpha < 1$). Then the inequalities in \eqref{eq16} and \eqref{eq17} become equality: 
    \begin{align}
        \alpha\mu_i &= -\gamma_{n-i+1}\quad \text{for}\quad\quad 1\leq i \leq \zeta^- \label{eq19}\\ \text{and } \quad \quad \rho_i &= \gamma_{i} \quad  \text{for} \quad 1 \leq i \leq \zeta^+. \label{eq20}
    \end{align}
    From \eqref{eq20}, we get $\rho_1=\gamma_1 = \gamma$ (say). We get from \cite{A_Q-merging_by_Nikiforov} that if $G$ is a connected graph with $\rho_1=\gamma_1$, then $G$ is $\gamma_1$-regular. Therefore $G$ is a connected $\gamma$-regular graph. Since $L(G) =\gamma I_n - A(G)$, we have $\mu_i = \gamma - \gamma_{n-i+1}$, $1\leq i \leq n$. In particular $\mu_1 = \gamma - \gamma_{n}$. Also from \eqref{eq19}, $\alpha\mu_1 = -\gamma_{n}$. Combining these two relation, we get $\alpha(\gamma -\gamma_n) = -\gamma_n$, i.e. $\gamma_n = -\frac{\alpha}{1-\alpha}\gamma$. This gives $|\gamma_n| \geq \gamma $, because $\frac{\alpha}{1-\alpha} \geq 1$ when $\frac{1}{2} \leq \alpha < 1$. But $|\gamma_n|$ can not be greater than the spectral radius $\gamma$ according to Perron-Frobenius theorem \cite{horn_matrix}, so $|\gamma_n|= \gamma $, i.e. $\gamma_n =-\gamma$ and $\alpha =\frac{1}{2}$. $\gamma_n =-\gamma=-\gamma_1$ implies that $G$ is bipartite (from \cite{th_of_graph_spec_by_cvetkovic}). Also as we have already seen that $G$ is regular, each partite set of $G$ is equal in size, and the order $n$ is even. If $G$ is any connected bipartite graph other than $K_{\frac{n}{2}, \frac{n}{2}}$ satisfying the above conditions, then it has diameter at least 3. Applying Lemma \ref{lem17} for $\alpha =0$, we can say that $G$ has at least $3+1$ i.e. $4$ distinct adjacency eigenvalues, which indicates $\zeta^+, \zeta^- \geq 2$ as adjacency eigenvalues are symmetrical about origin for a bipartite graph. So \eqref{eq19} and \eqref{eq20} hold for $i=2$. As it is already proved that $\alpha =\frac{1}{2}$ is necessary for the equality, so putting $\alpha =\frac{1}{2}$ and $i=2$ in \eqref{eq19}, we get $\mu_2 =-2\gamma_{n-1} =2\gamma_2$, the last equality holds due to the symmetry of adjacency eigenvalues of $G$. Also from $L(G)=\gamma I_n - A(G)$, we have $\mu_2=\gamma - \gamma_{n-1}$, which implies $\mu_2 = \gamma + \gamma_2$ similarly. Comparing last two expressions for $\mu_2$, we get $2\gamma_2 = \gamma + \gamma_2$ i.e. $\gamma_2 =\gamma$, which is a contradiction according to Perron-Frobenius theorem as $\gamma_1 =\gamma$. Therefore there exists no connected graph other than $K_{\frac{n}{2}, \frac{n}{2}}$ holding the equality and the proof finally concludes. \qedhere
\end{proof} 

\begin{theorem}\label{r23}
    Let $G$ be a graph with $s$ isolated vertices. If $0 \leq \alpha < 1$, then
    \begin{align}\label{eq104}
    \resizebox{.91\textwidth}{!}{$
         \big|E_{A_\alpha}(G) -(1-\alpha) \mathfrak{E}(\mathcal{L}(G)) - 2(1-\alpha) (n-m)\big| \leq |2\alpha -1|(2m-n+2s) +|n-2\alpha m|.
     $}    
    \end{align}
    Equality holds for $G \cong C_n$ and $\alpha = \frac{1}{2}$.
\end{theorem}

\begin{proof}
    We note $A_\alpha(G) -\frac{2\alpha m}{n}I_n = \alpha D(G) +(1-\alpha)A(G) -\frac{2\alpha m}{n}I_n=(1-\alpha)(L_S(G)-2I_n) + (2\alpha -1)(D(G)-I_n) +\big(1-\frac{2\alpha m}{n}\big)I_n$.
    Therefore
    \begin{align}\label{eq107}
        E_{A_\alpha}(G) & = \mathcal{E}\Big(A_\alpha(G) -\frac{2\alpha m}{n}I_n\Big)\notag \\
        &\leq (1-\alpha)\mathcal{E}\big(L_S(G)-2I_n\big) + |2\alpha -1|\mathcal{E}\big(D(G)-I_n\big) +\Big|1-\frac{2\alpha m}{n}\Big|\mathcal{E}(I_n)\notag \\
        &=(1-\alpha)\sum_{i=1}^{n}\big|q_i(G)-2\big| +|2\alpha -1|\sum_{i=1}^{n}\big|d_i(G)-1\big| + \Big|1-\frac{2\alpha m}{n}\Big|n,
    \end{align}
    by using Lemma \ref{lem9} and Lemma \ref{lem21}. Our claim is $\sum_{i=1}^{n}|d_i(G)-1|=2m-n+2s$. We consider different cases for all possible values of $s$. Let $X$ = $\sum_{i=1}^{n}|d_i(G)-1|$. For $s=0$, $d_i(G) \geq 1$, $i=1,2, \ldots, n.$ Therefore $X= \sum_{i=1}^{n}(d_i(G)-1)=2m -n$, as $\sum_{i=1}^{n}d_i(G)=2m$. Thus the required relation holds in this case. For $1\leq s\leq n-1$, $X=\sum_{i=1}^{n-s}(d_i(G)-1) + s= 2m-(n-s)+s=2m-n+2s$. When $s=n$, i.e. $G$ is a null graph on $n$ vertices, $X=\sum_{i=1}^{n}|0-1|=n$. Also $2m-n+2s=0-n+2n=n$. So the equality holds in this case too. Hence $\sum_{i=1}^{n}|d_i(G)-1|=2m-n+2s$. Applying this and the relation $\sum_{i=1}^{n}|q_i(G)-2| = \mathfrak{E}(\mathcal{L}(G))+2n-2m$ (from \cite{relation_by_kcdas}) in \eqref{eq107}, we get $E_{A_\alpha}(G) \leq (1-\alpha)\big(\mathfrak{E}(\mathcal{L}(G))+2n-2m\big) +|2\alpha-1|(2m-n+2s)+|n-2\alpha m|$.\par
    Rearranging the relation between $A_{\alpha}(G)$ and $L_S(G)$ by keeping only $(1-\alpha)(L_S(G)-2I_n)$ in the left hand side, then obtaining matrix energy in both sides and following the similar procedure as above, we get $E_{A_\alpha}(G) \geq (1-\alpha)\big(\mathfrak{E}(\mathcal{L}(G))+2n-2m\big) -|2\alpha -1|(2m-n+2s) - |n-2\alpha m|$.
    Combining the two inequalities obtained for $E_{\alpha}(G)$, we get the required result.

    Next we need to verify the equality condition. Substituting $\alpha = \frac{1}{2}$ and $G = C_n$, we obtain the left hand side of \eqref{eq104} as $| E_{A_{\frac{1}{2}}}(C_n) -\frac{1}{2}\mathfrak{E}(\mathcal{L}(C_n)) - 2\times \frac{1}{2} (n-n)|$, using $E_{A_{\frac{1}{2}}}(C_n) = \sum_{i=1}^n |\lambda_i(A_{\frac{1}{2}}(C_n)) - \frac{2\times \frac{1}{2}\times n}{n}| =\sum_{i=1}^n |\lambda_i(\frac{1}{2}\times 2I_n + \frac{1}{2}A(C_n)) - 1| = \sum_{i=1}^n |\frac{1}{2}\lambda_i(A(C_n))| = \frac{1}{2}\mathfrak{E}(C_n)$ and $\mathcal{L}(C_n) = C_n$, which becomes zero. Putting $\alpha = \frac{1}{2}$, $s=0$ and $m=n$ for $G=C_n$, the right hand side also reduces to zero. Hence the equality holds for $G \cong C_n$ and $\alpha = \frac{1}{2}$.
\end{proof}

By setting $\alpha = 0$ in Theorem \ref{r23}, we obtain the following corollary, which establishes a relation between adjacency energy of $G$ and that of its line graph $\mathcal{L}(G)$. The proof is immediate and thus omitted. 

\begin{corollary}\label{r104}
    Let $G$ be a graph with $s$ isolated vertices. Then
    \begin{align*}
         \big|\mathfrak{E}(G) -\mathfrak{E}(\mathcal{L}(G)) - 2 (n-m)\big| \leq 2(m+s).
    \end{align*}
\end{corollary}

Furthermore, assuming $G$ to be connected (so $s = 0$) yields the following corollary.

\begin{corollary}\label{r105}
    Let $G$ be a connected graph. Then
    \begin{align*}
         \big|\mathfrak{E}(G) -\mathfrak{E}(\mathcal{L}(G)) - 2 (n-m)\big| \leq 2m.
    \end{align*}
\end{corollary}

Using a similar proof technique as in Theorem \ref{r23}, the following theorem which involves one more basic graph invariant `number of pendant vertices' can be obtained easily.

\begin{theorem}\label{r24}
    Let $G$ be a graph with $s$ isolated vertices and $p$ pendant vertices. If $0 \leq \alpha < 1$, then
    \begin{align*}
         \resizebox{.95\textwidth}{!}{$\big|E_{A_\alpha}(G) -(1-\alpha) \mathfrak{E}(\mathcal{L}(G)) - 2(1-\alpha) (n-m)\big| \leq |2\alpha -1|(2m-2n+4s+2p) +2\alpha|n-m|$}.
    \end{align*}
    Equality is satisfied for $G \cong C_n$ for all $0 \leq \alpha < 1$.
\end{theorem}

With $\alpha =0$, and further by restricting $G$ to be a strictly binary tree (which has no isolated vertex and has exactly $\frac{n+1}{2}$ pendant vertices), we derive the next two corollaries directly from Theorem \ref{r24}.

\begin{corollary}
    Let $G$ be a graph with $s$ isolated vertices and $p$ pendant vertices. Then
    \begin{align*}
         \big|\mathfrak{E}(G) - \mathfrak{E}(\mathcal{L}(G)) - 2(n-m)\big| \leq (2m-2n+4s+2p).
    \end{align*}
    Equality is satisfied for $G \cong C_n$.
\end{corollary}

\begin{corollary}
    Let $G$ be a strictly binary tree on $n$ vertices. Then
    \begin{align*}
         \big|\mathfrak{E}(G) - \mathfrak{E}(\mathcal{L}(G)) - 2\big| \leq n-1.
    \end{align*}
\end{corollary}


\section{Concluding remarks}

We, by means of Theorem $\ref{r23}$ and Theorem $\ref{r24}$, introduce for the first time in the literature, a broad and explicit link between the general $A_\alpha$-energy of a graph and the adjacency energy of its line graph; a completely new bridge in the theory of graph energies with wide applicability across spectral graph theory. Although the theorems are stated as relations between two energies, they can be effectively used to produce lower and upper bounds for $E_{A_{\alpha}}(G)$, obtained by expanding modulus terms in the expressions, especially when direct computation of $E_{A_{\alpha}}(G)$ is difficult, whether due to the complexity of $G$ or its associated matrices, but $\mathcal{L}(G)$ is simpler or $\mathfrak{E}(\mathcal{L}(G))$ is easier to compute. 

In Theorems \ref{r11}, \ref{r13}, \ref{r17}, \ref{r23} and \ref{r24}, equality cases have been identified only for certain specific graph classes. Finding additional classes of graphs for which equality hold in these bounds (if exists), or establishing the converses, remain as open problems.


\section*{Statements and Declarations} 
\textbf{Competing interests:} The authors declare they have no competing interests.\par\noindent
\textbf{Availability of data and materials:} No data was used for the research described in the article.\par\noindent
\textbf{Funding information:} There is no funding source.


\bibliographystyle{plain}
\bibliography{biblioenergy}

@article{A_Q-merging_by_Nikiforov,
  AUTHOR = {Nikiforov, Vladimir},
     TITLE = {Merging the {$A$}- and {$Q$}-spectral theories},
   JOURNAL = {Appl. Anal. Discrete Math.},
  FJOURNAL = {Applicable Analysis and Discrete Mathematics},
    VOLUME = {11},
      YEAR = {2017},
    NUMBER = {1},
     PAGES = {81--107},
      ISSN = {1452-8630},
   MRCLASS = {05C50 (15A42)},
  MRNUMBER = {3648656},
MRREVIEWER = {Marianna Bolla},
       DOI = {10.2298/AADM1701081N},
       URL = {https://doi.org/10.2298/AADM1701081N},
}

@article {energy_by_gutman,
    AUTHOR = {Gutman, Ivan},
     TITLE = {The energy of a graph},
   JOURNAL = {Ber. Math.-Statist. Sekt. Forsch. Graz},
  FJOURNAL = {Graz. Forschungszentrum. Mathematisch-Statistische Sektion.
              Berichte},
    VOLUME = {103},
      YEAR = {1978},
     PAGES = {1--22},
   MRCLASS = {05C50},
  MRNUMBER = {525890},
MRREVIEWER = {A.\ J.\ Schwenk},
}

@article {lap_energy_by_gutman,
    AUTHOR = {Gutman, Ivan and Zhou, Bo},
     TITLE = {Laplacian energy of a graph},
   JOURNAL = {Linear Algebra Appl.},
  FJOURNAL = {Linear Algebra and its Applications},
    VOLUME = {414},
      YEAR = {2006},
    NUMBER = {1},
     PAGES = {29--37},
      ISSN = {0024-3795,1873-1856},
   MRCLASS = {05C50 (68R10)},
  MRNUMBER = {2209232},
MRREVIEWER = {Irene\ N. M. Sciriha},
       DOI = {10.1016/j.laa.2005.09.008},
       URL = {https://doi.org/10.1016/j.laa.2005.09.008},
}

@article {signless_lap_energy_by_abreu,
    AUTHOR = {Abreu, Nair and Cardoso, Domingos M. and Gutman, Ivan and
              Martins, Enide A. and Robbiano, Mar\'ia},
     TITLE = {Bounds for the signless {L}aplacian energy},
   JOURNAL = {Linear Algebra Appl.},
  FJOURNAL = {Linear Algebra and its Applications},
    VOLUME = {435},
      YEAR = {2011},
    NUMBER = {10},
     PAGES = {2365--2374},
      ISSN = {0024-3795,1873-1856},
   MRCLASS = {05C50 (15A18)},
  MRNUMBER = {2811121},
MRREVIEWER = {Muhuo\ Liu},
       DOI = {10.1016/j.laa.2010.10.021},
       URL = {https://doi.org/10.1016/j.laa.2010.10.021},
}

@article {A_alpha_energy_by_guo,
    AUTHOR = {Guo, Haiyan and Zhou, Bo},
     TITLE = {On the {$\alpha$}-spectral radius of graphs},
   JOURNAL = {Appl. Anal. Discrete Math.},
  FJOURNAL = {Applicable Analysis and Discrete Mathematics},
    VOLUME = {14},
      YEAR = {2020},
    NUMBER = {2},
     PAGES = {431--458},
      ISSN = {1452-8630,2406-100X},
   MRCLASS = {05C50},
  MRNUMBER = {4201153},
MRREVIEWER = {Shiping\ Liu},
       DOI = {10.2298/aadm180210022g},
       URL = {https://doi.org/10.2298/aadm180210022g},
}

@article {sum_of_A_alpha_eigenvalues_by_lin,
    AUTHOR = {Lin, Zhen},
     TITLE = {On the sum of the largest {$A_\alpha$}-eigenvalues of graphs},
   JOURNAL = {AIMS Math.},
  FJOURNAL = {AIMS Mathematics},
    VOLUME = {7},
      YEAR = {2022},
    NUMBER = {8},
     PAGES = {15064--15074},
      ISSN = {2473-6988},
   MRCLASS = {05C50},
  MRNUMBER = {4443427},
       DOI = {10.3934/math.2022825},
       URL = {https://doi.org/10.3934/math.2022825},
}

@article {A_alpha_energy_and_zagreb_by_pirzada,
    AUTHOR = {Pirzada, S. and Rather, Bilal A. and Ganie, Hilal A. and ul
              Shaban, Rezwan},
     TITLE = {On {$\alpha $}-adjacency energy of graphs and {Z}agreb index},
   JOURNAL = {AKCE Int. J. Graphs Comb.},
  FJOURNAL = {AKCE International Journal of Graphs and Combinatorics},
    VOLUME = {18},
      YEAR = {2021},
    NUMBER = {1},
     PAGES = {39--46},
      ISSN = {0972-8600,2543-3474},
   MRCLASS = {05C50 (05C09 05C12 15A18)},
  MRNUMBER = {4277578},
       DOI = {10.1080/09728600.2021.1917973},
       URL = {https://doi.org/10.1080/09728600.2021.1917973},
}

@article {bounds_on_A_alpha_energy_by_zhou,
    AUTHOR = {Zhou, Lianlian and Li, Dan and Chen, Yuanyuan and Meng,
              Jixiang},
     TITLE = {Some bounds on the {$A_\alpha$}-energy of graphs},
   JOURNAL = {Filomat},
  FJOURNAL = {Univerzitet u Ni\v su. Prirodno-Matemati\v cki Fakultet.
              Filomat},
    VOLUME = {38},
      YEAR = {2024},
    NUMBER = {4},
     PAGES = {1329--1341},
      ISSN = {0354-5180,2406-0933},
   MRCLASS = {05C50},
  MRNUMBER = {4678456},
}

@article {th_of_weyl_by_fan,
    AUTHOR = {Fan, Ky},
     TITLE = {On a theorem of {W}eyl concerning eigenvalues of linear
              transformations. {I}},
   JOURNAL = {Proc. Nat. Acad. Sci. U.S.A.},
  FJOURNAL = {Proceedings of the National Academy of Sciences of the United
              States of America},
    VOLUME = {35},
      YEAR = {1949},
     PAGES = {652--655},
      ISSN = {0027-8424},
   MRCLASS = {46.3X},
  MRNUMBER = {34519},
MRREVIEWER = {F.\ Smithies},
       DOI = {10.1073/pnas.35.11.652},
       URL = {https://doi.org/10.1073/pnas.35.11.652},
}

@article {max_props_by_fan,
    AUTHOR = {Fan, Ky},
     TITLE = {Maximum properties and inequalities for the eigenvalues of
              completely continuous operators},
   JOURNAL = {Proc. Nat. Acad. Sci. U.S.A.},
  FJOURNAL = {Proceedings of the National Academy of Sciences of the United
              States of America},
    VOLUME = {37},
      YEAR = {1951},
     PAGES = {760--766},
      ISSN = {0027-8424},
   MRCLASS = {46.3X},
  MRNUMBER = {45952},
MRREVIEWER = {L.\ G\aa rding},
       DOI = {10.1073/pnas.37.11.760},
       URL = {https://doi.org/10.1073/pnas.37.11.760},
}

@book {analytic_inequalities_by_mitrinovic,
    AUTHOR = {Mitrinovi\'c, D. S.},
     TITLE = {Analytic inequalities},
    SERIES = {Die Grundlehren der mathematischen Wissenschaften},
    VOLUME = {Band 165},
      NOTE = {In cooperation with P. M. Vasi\'c},
 PUBLISHER = {Springer-Verlag, New York-Berlin},
      YEAR = {1970},
     PAGES = {xii+400},
   MRCLASS = {26.70},
  MRNUMBER = {274686},
MRREVIEWER = {R.\ P.\ Boas},
}

@article {bounds_on_A_alpha_spread_by_lin,
    AUTHOR = {Lin, Zhen and Miao, Lianying and Guo, Shu-Guang},
     TITLE = {Bounds on the {$A_\alpha$}-spread of a graph},
   JOURNAL = {Electron. J. Linear Algebra},
  FJOURNAL = {Electronic Journal of Linear Algebra},
    VOLUME = {36},
      YEAR = {2020},
     PAGES = {214--227},
      ISSN = {1081-3810},
   MRCLASS = {05C50},
  MRNUMBER = {4094557},
MRREVIEWER = {Juan\ P.\ Rada},
       DOI = {10.13001/ela.2020.5137},
       URL = {https://doi.org/10.13001/ela.2020.5137},
}

@article {bounds_for_A_alpha_spec_rad_by_alhevaz,
    AUTHOR = {Alhevaz, Abdollah and Baghipur, Maryam and Ganie, Hilal Ahmad},
     TITLE = {Bounds for the spectral radius of the {$A_{\alpha }$}-matrix
              of graphs},
   JOURNAL = {Indian J. Pure Appl. Math.},
  FJOURNAL = {Indian Journal of Pure and Applied Mathematics},
    VOLUME = {55},
      YEAR = {2024},
    NUMBER = {1},
     PAGES = {298--309},
      ISSN = {0019-5588,0975-7465},
   MRCLASS = {05C50 (05C12)},
  MRNUMBER = {4703149},
       DOI = {10.1007/s13226-023-00363-9},
       URL = {https://doi.org/10.1007/s13226-023-00363-9},
}

@article {commutativity_by_so,
    AUTHOR = {So, Wasin},
     TITLE = {Commutativity and spectra of {H}ermitian matrices},
   JOURNAL = {Linear Algebra Appl.},
  FJOURNAL = {Linear Algebra and its Applications},
    VOLUME = {212/213},
      YEAR = {1994},
     PAGES = {121--129},
      ISSN = {0024-3795,1873-1856},
   MRCLASS = {15A42 (15A57)},
  MRNUMBER = {1306975},
MRREVIEWER = {Yik-Hoi\ Au-Yeung},
       DOI = {10.1016/0024-3795(94)90399-9},
       URL = {https://doi.org/10.1016/0024-3795(94)90399-9},
}

@book {th_of_graph_spec_by_cvetkovic,
    AUTHOR = {Cvetkovi\'c, Drago\'s and Rowlinson, Peter and Simi\'c,
              Slobodan},
     TITLE = {An introduction to the theory of graph spectra},
    SERIES = {London Mathematical Society Student Texts},
    VOLUME = {75},
 PUBLISHER = {Cambridge University Press, Cambridge},
      YEAR = {2010},
     PAGES = {xii+364},
      ISBN = {978-0-521-13408-8},
   MRCLASS = {05-02 (05C50)},
  MRNUMBER = {2571608},
MRREVIEWER = {Ligong\ Wang},
}

@article {relation_by_kcdas,
    AUTHOR = {Das, Kinkar Ch. and Mojallal, Seyed Ahmad},
     TITLE = {Relation between signless {L}aplacian energy, energy of graph
              and its line graph},
   JOURNAL = {Linear Algebra Appl.},
  FJOURNAL = {Linear Algebra and its Applications},
    VOLUME = {493},
      YEAR = {2016},
     PAGES = {91--107},
      ISSN = {0024-3795,1873-1856},
   MRCLASS = {05C50},
  MRNUMBER = {3452728},
MRREVIEWER = {Durmu\c s\ Bozkurt},
       DOI = {10.1016/j.laa.2015.12.006},
       URL = {https://doi.org/10.1016/j.laa.2015.12.006},
}

@article {spectra_of_join_by_basunia,
    AUTHOR = {Basunia, Mainak and Mahato, Iswar and Kannan, M. Rajesh},
     TITLE = {On the {$A _\alpha$}-spectra of some join graphs},
   JOURNAL = {Bull. Malays. Math. Sci. Soc.},
  FJOURNAL = {Bulletin of the Malaysian Mathematical Sciences Society},
    VOLUME = {44},
      YEAR = {2021},
    NUMBER = {6},
     PAGES = {4269--4297},
      ISSN = {0126-6705,2180-4206},
   MRCLASS = {05C50 (05C05)},
  MRNUMBER = {4321762},
MRREVIEWER = {Jianxi\ Li},
       DOI = {10.1007/s40840-021-01166-z},
       URL = {https://doi.org/10.1007/s40840-021-01166-z},
}

@article {note_on_positive_semidefiniteness_by_nikiforov,
    AUTHOR = {Nikiforov, Vladimir and Rojo, Oscar},
     TITLE = {A note on the positive semidefiniteness of {$A_\alpha(G)$}},
   JOURNAL = {Linear Algebra Appl.},
  FJOURNAL = {Linear Algebra and its Applications},
    VOLUME = {519},
      YEAR = {2017},
     PAGES = {156--163},
      ISSN = {0024-3795,1873-1856},
   MRCLASS = {05C50 (05C69 15B48)},
  MRNUMBER = {3606267},
MRREVIEWER = {Craig\ J.\ Erickson},
       DOI = {10.1016/j.laa.2016.12.042},
       URL = {https://doi.org/10.1016/j.laa.2016.12.042},
}

@article {A_alpha_spec_of_graphs_by_lin_xue_shu,
    AUTHOR = {Lin, Huiqiu and Xue, Jie and Shu, Jinlong},
     TITLE = {On the {$A_\alpha$}-spectra of graphs},
   JOURNAL = {Linear Algebra Appl.},
  FJOURNAL = {Linear Algebra and its Applications},
    VOLUME = {556},
      YEAR = {2018},
     PAGES = {210--219},
      ISSN = {0024-3795,1873-1856},
   MRCLASS = {05C50 (05C12)},
  MRNUMBER = {3842581},
       DOI = {10.1016/j.laa.2018.07.003},
       URL = {https://doi.org/10.1016/j.laa.2018.07.003},
}

@article{On_max_a_alpha_spec_rad_of_unicyclic_bicyclic_fixed_girth_pendant_by_das_mahato,
  title={On the maximum {$A_\alpha$}-spectral radius of unicyclic and bicyclic graphs with fixed girth or fixed number of pendant vertices},
  author={Das, Joyentanuj and Mahato, Iswar},
  journal={Comput. Appl. Math.},
  volume={43},
  number={6},
  pages={347},
  year={2024},
  publisher={Springer}
}

@article {bounds_for_A_alpha_energy_by_shaban,
    AUTHOR = {Shaban, Rezwan Ul and Imran, Muhammad and Ganie, Hilal A.},
     TITLE = {Bounds for the ;1-adjacency energy of a graph},
   JOURNAL = {J. Math. Inequal.},
  FJOURNAL = {Journal of Mathematical Inequalities},
    VOLUME = {18},
      YEAR = {2024},
    NUMBER = {1},
     PAGES = {127--141},
      ISSN = {1846-579X,1848-9575},
   MRCLASS = {05C50 (05C12 15A18)},
  MRNUMBER = {4764231},
       DOI = {10.7153/jmi-2024-18-08},
       URL = {https://doi.org/10.7153/jmi-2024-18-08},
}

@article {energy_of_line_graphs_by_gutman_robbiano,
    AUTHOR = {Gutman, Ivan and Robbiano, Mar\'ia and Martins, Enide Andrade
              and Cardoso, Domingos M. and Medina, Luis and Rojo, Oscar},
     TITLE = {Energy of line graphs},
   JOURNAL = {Linear Algebra Appl.},
  FJOURNAL = {Linear Algebra and its Applications},
    VOLUME = {433},
      YEAR = {2010},
    NUMBER = {7},
     PAGES = {1312--1323},
      ISSN = {0024-3795,1873-1856},
   MRCLASS = {05C50 (05C62 15A18)},
  MRNUMBER = {2680258},
MRREVIEWER = {Zoran\ S.\ Radosavljevi\'c},
       DOI = {10.1016/j.laa.2010.05.009},
       URL = {https://doi.org/10.1016/j.laa.2010.05.009},
}

@article {laplacian_energy_of_graphs_by_kcdas,
    AUTHOR = {Das, Kinkar Ch. and Mojallal, Seyed Ahmad},
     TITLE = {On {L}aplacian energy of graphs},
   JOURNAL = {Discrete Math.},
  FJOURNAL = {Discrete Mathematics},
    VOLUME = {325},
      YEAR = {2014},
     PAGES = {52--64},
      ISSN = {0012-365X,1872-681X},
   MRCLASS = {05C50},
  MRNUMBER = {3181233},
MRREVIEWER = {Renata\ Raposo\ Del Vecchio},
       DOI = {10.1016/j.disc.2014.02.017},
       URL = {https://doi.org/10.1016/j.disc.2014.02.017},
}

@article {upper_bound_on_sum_square_degree_by_caen,
    AUTHOR = {de Caen, D.},
     TITLE = {An upper bound on the sum of squares of degrees in a graph},
   JOURNAL = {Discrete Math.},
  FJOURNAL = {Discrete Mathematics},
    VOLUME = {185},
      YEAR = {1998},
    NUMBER = {1-3},
     PAGES = {245--248},
      ISSN = {0012-365X,1872-681X},
   MRCLASS = {05C35},
  MRNUMBER = {1614258},
       DOI = {10.1016/S0012-365X(97)00213-6},
       URL = {https://doi.org/10.1016/S0012-365X(97)00213-6},
}

@book{horn_matrix,
  title     = {{M}atrix {A}nalysis},
  author    = {Horn, Roger A. and Johnson, Charles R.},
  edition   = {2nd},
  year      = {2012},
  publisher = {Cambridge University Press},
  address   = {Cambridge}
}
\end{document}